\documentclass[12pt,reqno]{amsart}
\usepackage{amsmath,amssymb,amscd,verbatim,color}

\makeatletter
\@namedef{subjclassname@2020}{%
	\textup{2020} Mathematics Subject Classification}
\makeatother
\usepackage[colorlinks=true, linkcolor=blue]{hyperref}
\usepackage{amsfonts}
\usepackage[top=35mm, bottom=35mm, left=30mm, right=30mm]{geometry}

\usepackage{xcolor}
\theoremstyle{plain}
\newtheorem{thm}{Theorem}[section]
\newtheorem{cor}[thm]{Corollary}

\newtheorem{lem}[thm]{Lemma}
\newtheorem{prop}[thm]{Proposition}

\theoremstyle{definition}
\newtheorem{defn}[thm]{Definition}
\newtheorem{exam}[thm]{Example}
\newtheorem{rem}[thm]{Remark}

\numberwithin{equation}{section}

\newcommand{\mb}{\mathbb}
\newcommand{\mc}{\mathcal}
\newcommand{\lk}{\left}
\newcommand{\re}{\right}

 \def \b{\beta}  
 \def \T{\Theta}   
   \def \o{\omega}
\def \O{\Omega}

\allowdisplaybreaks

\begin{document}
	
	\title[ Directional weak mixing for $\mathbb{Z}^q$-actions ]{Directional weak mixing for $\mathbb{Z}^q$-actions}

	\author[C. Liu]{Chunlin Liu}
	\address{C. Liu: CAS Wu Wen-Tsun Key Laboratory of Mathematics, School of Mathematical Sciences, University of Science and Technology of China, Hefei, Anhui, 230026, P.R. China}
	
	\email{lcl666@mail.ustc.edu.cn}

	\subjclass[2020]{Primary  37A25, 37A35; Secondary  37A05}
	
	\keywords{Directional sequence entropy, Directional weak mixing, Directional Kronecker algebra.}
	\begin{abstract}
		In this paper, we introduce and characterize the concept of directional weak mixing through independence, sequence entropy, the mean ergodic theorem, and other notions. Additionally, we deduce a directional version of the Koopman-von Neumann spectrum mixing theorem. Furthermore, we explore the relation  between directional weak mixing and weak mixing.
	\end{abstract}

	\maketitle
	
	\section{Introduction}
	Various mixing properties play a crucial role in ergodic theory, and researchers have approached their study from diverse perspectives. Notably, characterizing weak mixing through sequence entropy has proven to be particularly effective and fruitful. The pioneering work was originated by  Ku\v shnirenko \cite{Kus}, who introduced measure-theoretical sequence entropy  for $\mathbb{Z}$-systems and proved that a $\mathbb{Z}$-system has discrete spectrum if and only if its sequence entropy is zero  with respect to any infinite subsequence of $\mathbb{Z}_+$. Later, Saleski \cite{Sa} gave a characterization of weak mixing  via sequence entropy. Hulse \cite{H,Hu1} improved some of the Saleski's results. Furthermore, Huang, Shao and Ye \cite{HSY} investigated kinds of mixing via sequence entropy. Zhang \cite{MR1152677,MR1208125} investigated mild mixing and mildly mixing extension via sequence entropy. Recently, Coronel, Maass and Shao \cite{CMS} studied Kronecker and rigid  algebras by sequence entropy. More recently, Liu and Yan \cite{Liu_2023} extended some of  above results to systems under amenable group actions.
	
	To explore the Cellular Automaton map together with the Bernoulli shift, Milnor \cite{M1,M} introduced directional entropy.  Subsequently, Park \cite{P} demonstrated the continuity of directional entropy for a $\mathbb{Z}^2$-action generated by a cellular automaton. With the deepening of research, mathematicians
	are not only focused on the directional entropy of cellular automaton but on more general systems as well. More properties were further studied in \cite{MR3551897,MR3681987,KP,liu2023directional,MR1355676,Park111,MR2755932,2022arXiv220306710R}. A natural question arises: Can directional sequence entropy be employed, similar to the $\mathbb{Z}$-system, to investigate the spectral and mixing properties along a specific direction? To investigate this question, Liu and Xu \cite{LX} introduced  notions of directional sequence entropy and directional discrete spectrum, and characterized directional discrete spectrum via directional sequence entropy.  Moreover,  they \cite{LX2} also characterized directional discrete spectrum by directional complexity. From the Koopman-von Neumann spectrum mixing theorem  \cite{KN}, it can be observed that in 
	$\mathbb{Z}$-systems,  discrete spectrum is the opposite extreme to weak mixing. Thus, based on ideas in \cite{LX},  in this paper, we introduce and study the notion of directional weak mixing.

	To be precise, throughout this paper, by a $\mathbb{Z}^2$-measure preserving dynamical system ($\mathbb{Z}^2$-m.p.s. for short) we mean a quadruple $(X,\mathcal{B}_X,\mu, T)$,
	where $X$ is a compact metric space, the $\mathbb{Z}^2$-action $T$ is a homeomorphism from the additive group $\mathbb{Z}^2$ to the group of homeomorphisms of $X$, $\mathcal{B}_X$ is the Borel $\sigma$-algebra of $X$ and $\mu$ is  an invariant Borel probability measure.  We denote the corresponding homeomorphism on $X$ by $T^{(m,n)}$ for any $(m,n)\in \mathbb{Z}^2$, so that $T^{(m_1,n_1)}\circ T^{(m_2,n_2)}=T^{(m_1+m_2,n_1+n_2)}$ and $T^{(0,0)} $ is the identity on $X$.
	The direction vector is denoted by  $\vec{v}=(1,\beta)\in \mb{R}^2$. We always assume that $\vec{v} \notin \mathbb{Q}$ because the cases where $\vec{v} \in \mathbb{Q}$ or $\vec{v} = (0,1)$ can be regarded as a $\mathbb{Z}$-m.p.s., yielding classical results. Moreover, our paper specifically concentrates on the $\mathbb{Z}^2$-m.p.s. However, we can verify that these results also hold for $\mathbb{Z}^q$-m.p.s., where $q$ is any integer greater than 2.
	
	We put $$\Lambda^{\vec{v}}(b)=\left\{(m,n)\in\mathbb{Z}^2:\beta m-b\leq n\leq \beta m+b\right\},$$
	write $\Lambda_+^{\vec{v}}(b)=\Lambda^{\vec{v}}(b)\cap (\mathbb{Z}_+\times \mathbb{Z})$  and $\Lambda_k^{\vec{v}}(b)=\Lambda^{\vec{v}}(b)\cap ([0,k-1]\times \mathbb{Z})$ for each $k\in\mathbb{N}$.
	Define the directional Kronecker algebra along the direction $\vec{v}$ by
	$$\mathcal{K}_\mu^{\vec{v}}=\left\{B\in\mathcal{B}_X: \overline{\lk\{U_T^{(m,n)}1_B :(m,n)\in \Lambda^{\vec{v}}(b) \re\}}\text{ is compact in } L^2(X,\mathcal{B}_X,\mu) \right\},$$
	where $U_T^{(m,n)}:L^2(X,\mathcal{B}_X,\mu)\to L^2(X,\mathcal{B}_X,\mu)$ is the unitary operator such that
	$$U_T^{(m,n)}f=f\circ T^{(m,n)}\text{ for all }f\in L^2(X,\mathcal{B}_X,\mu)$$ and $1_B$ is the indicator function of $B\in \mathcal{B}_X$.
	Note that the definition of $\mathcal{K}_\mu^{\vec{v}}$ is independent of the choice of $b\in(0,\infty)$ (see \cite[Proposition 3.1]{LX}). Analogous to $\mathbb{Z}$-actions, we call a $\mathbb{Z}^2$-m.p.s. $(X ,\mathcal{B}_X, \mu, T)$ $\vec{v}$-weak mixing if $\mathcal{K}_\mu^{\vec{v}}=\{X,\emptyset\}$. The following result is a characterization of directional weak mixing   via independence.
	\begin{thm}
		For any $b\in(0,\infty)$, the following statements are equivalent.
		\begin{itemize}
			\item [(a)] $(X,\mathcal{B}_X, \mu, T)$ is $\vec{v}$-weakly mixing.
			\item [(b)]  For any $f\in L^2(X,\mathcal{B}_X,\mu)$, $$\lim_{k\rightarrow \infty}\frac{1}{\#(\Lambda_k^{\vec{v}}(b))}\sum_{(m,n)\in\Lambda_k^{\vec{v}}(b)}|\langle U_T^{(m,n)}f,f\rangle-\langle f,1\rangle\langle 1,f\rangle|=0.$$
			\item [(c)] 
			For any $f,g\in L^2(X,\mathcal{B}_X,\mu)$, $$\lim_{k\rightarrow \infty}\frac{1}{\#(\Lambda_k^{\vec{v}}(b))}\sum_{(m,n)\in\Lambda_k^{\vec{v}}(b)}|\langle U_T^{(m,n)}f,g\rangle-\langle f,1\rangle\langle1,g\rangle|=0.$$
		\end{itemize}
	\end{thm}
	
	Next, we characterize directional weak mixing  via directional sequence entropy. Let us begin with some notations.
	For a finite measurable partition $\alpha$ of $X$, let $H_{\mu}(\alpha)=-\sum_{A\in \alpha}\mu(A)\log{\mu(A)}.$ 
	The following result can be found in Page 94 of \cite{Peter}.
	\begin{lem}\label{lem-2}Let $(X,\mathcal{B}_X,\mu)$ be a Borel probability space and $r\geq 1$ be a fixed integer. For each $\epsilon >0$, there exists  $\delta=\delta(\epsilon,r)>0$ such that if  $\alpha=\{A_1,A_2,\ldots,A_r\}$ and $\eta=\{B_1,B_2,\ldots,B_r\}$ are any two finite measurable partitions of $(X,\mathcal{B}_X,\mu)$ with $\sum_{j=1}^r\mu(A_j\Delta B_j)< \delta$ then $H_\mu(\alpha|\eta)+H_\mu(\eta|\alpha)<\epsilon$.
	\end{lem}
	For any infinite sequence $S=\{(m_i,n_i)\}_{i=1}^{\infty}$ of $\Lambda^{\vec{v}}(b)$, we put
	$$ h^S_{\mu}(T,\alpha)= \limsup_{k\to \infty} \frac{1}{k}H_\mu \lk(\bigvee_{i=1}^k T^{-(m_i,n_i)} \alpha\re).$$
	Then we can define the directional sequence entropy of $T$ for the infinite subset $S$ by
	$$h^S_{\mu}(T)=\sup_{\alpha}h^S_{\mu}(T,\alpha),$$
	where the supremum is taken over all finite measurable partitions $\alpha$ of $X$. 
	Motivated by Saleski's work \cite{Sa}, we characterize directional weak mixing  by directional sequence entropy.
	\begin{thm}
		For any $b\in(0,\infty)$, the following three conditions are equivalent.
		\begin{itemize}
			\item[(a)]	$(X,\mathcal{B}_X,\mu,T)$ is  $\vec{v}$-weakly mixing.	
			\item[(b)]Given nontrivial finite measurable partition $\alpha$ of $X$, there exists an infinite sequence $S=\{(m_i,n_i)\}_{i=1}^{\infty}$ of $\Lambda^{\vec{v}}(b)$ such that  $h^{S}_{\mu}(T,\alpha)>0.$
		\end{itemize}
	\end{thm}
	Moreover, we can choose a common subsequence of $\Lambda_+^{\vec{v}}(b)$ for all finite measurable partitions.
	\begin{thm}
		For any $b\in(0,\infty)$, the following two conditions are equivalent.
		\begin{itemize}
			\item[(a)]	$(X,\mathcal{B}_X,\mu,T)$ is  $\vec{v}$-weakly mixing.	
			\item[(b)]There exists an infinite sequence $S=\{(m_i,n_i)\}_{i=1}^{\infty}$ of $\Lambda^{\vec{v}}(b)$ such that  $h^{S}_{\mu}(T,\alpha)=H_{\mu}(\alpha)$
			for any finite measurable partition $\alpha$ of $X$.
		\end{itemize}
	\end{thm}

	To further investigate the properties of directional weak mixing, we introduce directional weak mixing functions. A function $f\in L^2(X,\mathcal{B}_X,\mu)$ is said to be $\vec{v}$-weakly mixing, if $\mathbb{E}(f|\mathcal{K}_\mu^{\vec{v}})(x)=0$ for $\mu$-a.e. $x\in X$. Denote by $WM^{\mu,\vec{v}}(X)$ the set of $\vec{v}$-weakly mixing functions for $\mu$, and we obtain a  directional version of Koopman-von Neumann spectrum mixing theorem \textbf{(see Theorem \ref{thm8})}:
	$$L^2(X,\mathcal{B}_X,\mu)= WM^{\mu,\vec{v}}(X)\bigoplus L^2(X,\mathcal{K}_\mu^{\vec{v}},\mu).$$
	Moreover, we provide a characterization of weakly mixing functions \textbf{(see Theorem \ref{3.10})}. Namely
	\begin{equation*}
		\begin{split}
			&WM^{\mu,\vec{v}}(X)=\{f\in L^2(X,\mathcal{B}_X,\mu):\\
			&\qquad\qquad \text{ for any }g\in L^2(X,\mathcal{B}_X,\mu), \lim_{k\to \infty}\frac{1}{\#(\Lambda_k^{\vec{v}}(b))}\sum_{(m,n)\in\Lambda_k^{\vec{v}}(b)}|\langle U_T^{(m,n)}f,g\rangle|=0\}.
		\end{split}
	\end{equation*}
	The mean ergodic theorem plays an important role in ergodic theory. Consequently, we also characterize directional weak mixing  through the application of a  directional version of mean ergodic theorem.
	For convenience, we introduce the following notations.
	Let $Q$ be an infinite subset of $\Lambda_+^{\vec{v}}(b)$, and define the upper  and lower density of $Q$ in $\Lambda_+^{\vec{v}}(b)$, respectively, by
	$$\overline{D}_b(Q)=\limsup_{k\to\infty}\frac{\#(Q\cap([0,k-1]\times \mathbb{Z}))}{\#(\Lambda_k^{\vec{v}}(b))}$$
	and
	$$\underline{D}_b(Q)=\liminf_{k\to\infty}\frac{\#(Q\cap([0,k-1]\times \mathbb{Z}))}{\#(\Lambda_k^{\vec{v}}(b))}.$$
	If $\overline{D}_b(Q)=\underline{D}_b(Q)$, we say $Q$ has the density, and the common value is denoted by $D_b(Q)$.
	\begin{thm}
		For any $b\in(0,\infty)$, the following conditions are equivalent. 
		\begin{itemize}
			\item[(a)]	$(X,\mathcal{B}_X,\mu,T)$ is $\vec{v}$-weakly mixing.	
			\item[(b)]For any infinite subset $Q=\{(m_i,n_i)\}_{i=1}^\infty$ of $\Lambda^{\vec{v}}(b)$  with $\underline{D}_b(Q)>0$,
			one has 
			\begin{equation*}
				\lim_{N\to \infty}\|\frac{1}{N}\sum_{i=1}^{N}U_T^{(m_i,n_i)}f-\int_{X}fd\mu\|_2=0
			\end{equation*}
			for all $f\in L^2(X,\mathcal{B}_X,\mu)$.
		\end{itemize}
	\end{thm}

	\medskip
	The structure of the paper is as follows. In Section \ref{Pre}, we recall some basic notions that we use in this paper. In Section \ref{sec:weak mixing}, we introduce  directional weak mixing, and characterize it through various facets.  In Section 4, we provide a directional version of the mean ergodic theorem.
	In Section \ref{sec:relation}, we discuss the relation between weak mixing and directional weak mixing.

	\section{Preliminaries}\label{Pre}
	In this section we recall some notations and results that are used later.
	\subsection{General notions.}
	In this paper, the sets of real numbers, rational numbers, integers, non-negative integers, and natural numbers are denoted by $\mathbb{R}$, $\mathbb{Q}$, $\mathbb{Z}$, $\mathbb{Z}_+$, and $\mathbb{N}$, respectively. We use $\overline{E}$ to denote the closure of a subset $E$ of $X$ and $|f|_p$ to denote the $L^p$-norm of a function $f$ defined in a Borel probability measure space $(X, \mathcal{B}_X, \mu)$, where
	$$ \|f\|_p=  
	(\int_{X}|f|^pd\mu)^{\frac{1}{p}}, \text{ if }  1\leq p<\infty \text{ and } \|f\|_{\infty}=\inf\{a\geq 0:\mu\left(\{x:|f(x)|>a\}\right)=0\}.$$

	\subsection{Kronecker algebra and discrete spectrum.}
	\subsubsection{Kronecker algebra and discrete spectrum for $\mathbb{Z}$-actions.}
	The content in this subsection can be found in many books, such as \cite{HostKrabook}. Let $(X,\mathcal{B}_X,\mu,T)$ be a $\mathbb{Z}$-m.p.s. and $\mathcal{H}=L^2(X,\mathcal{B}_X,\mu)$. In the complex Hilbert space $\mathcal{H}$, we define the unitary operator $U_T:\mathcal{H}\rightarrow \mathcal{H}$ by $U_Tf=f\circ T$, for any $f\in\mathcal{H}$.  We say that $f$ is an almost periodic function if $\overline{\left\{U_T^n f:n\in \mathbb{Z}\right\}}$ is a compact subset of $\mathcal{H}$.
	It is well known that the set of all bounded almost periodic functions forms a $U_T$-invariant and conjugation-invariant subalgebra of $\mathcal{H}$ (denoted by $\mathcal{A}_c$). The set of almost periodic functions is just the closure of $\mathcal{A}_c$ (denoted by $\mathcal{H}_c$). By \cite[Theorem 1.2]{Zi}, there exists a $T$-invariant sub-$\sigma$-algebra $\mathcal{K}_{\mu}$ of $\mathcal{B}_X$ such that $\mathcal{H}_c=L^2(X,\mathcal{K}_{\mu},\mu)$. The sub-$\sigma$-algebra $\mathcal{K}_{\mu}$ is called the Kronecker algebra of $(X,\mathcal{B}_X,\mu,T)$.    It is easy to know that $\mathcal{K}_\mu$ consists of all $B\in \mathcal{B}_X$ such that
	$\overline{\{ U_T^n1_B:n\in \mathbb{Z} \}} \text{ is compact in }L^2(X,\mathcal{B}_X,\mu).$
	We say $\mu$ has discrete spectrum if $\mathcal{B}_X=\mathcal{K}_{\mu}$. 	The following result is classical.
	\begin{lem}\label{lem5}
		Let  $(X_1,\mathcal{B}_{X_1},\mu,T_1)$ and $(X_2,\mathcal{B}_{X_2},\nu,T_2)$ be two $\mathbb{Z}$-m.p.s. Then $$\mathcal{K}_{\mu\times \nu}(X_1\times X_2)=\mathcal{K}_{\mu}(X_1)\times \mathcal{K}_{\nu}(X_2).$$
	\end{lem}
	
	\subsubsection{Directional Kronecker algebra and discrete spectrum.}
	In this subsection, we recall the directional Kronecker algebra and discrete spectrum introduced by Liu and Xu  \cite{LX}.
	Let $(X,\mathcal{B}_X,\mu,T)$ be a $\mathbb{Z}^2$-m.p.s., $\vec{v}=(1,\beta)\in\mathbb{R}^2$ be a direction vector and $b\in (0,\infty)$. 
	Set $$\Lambda^{\vec{v}}(b)=\left\{(m,n)\in\mathbb{Z}^2:\beta m-b/2\leq n\leq \beta m+b/2\right\}.$$
	Let $\mathcal{A}_c^{\vec{v}}(b)$ be the collection of $f\in \mathcal{H}=L^2(X,\mathcal{B}_X,\mu)$ such that
	$$\overline{\left\{U_T^{(m,n)}f:(m,n)\in \Lambda^{\vec{v}}(b) \right\}}\text{ is compact in }L^2(X,\mathcal{B}_X,\mu).$$
	It is easy to see that  $\mathcal{A}_c^{\vec{v}}(b)$ is a $U_{T^{\vec{w}}}$-invariant for any $\vec{w}$ in $\mathbb{Z}^2$ and conjugation-invariant subalgebra of $\mathcal{H}$. It follows from  \cite[Theorem 1.2]{Zi} that there exists a $T$-invariant sub-$\sigma$-algebra  $\mathcal{K}_\mu^{\vec{v}}(b)$  of $\mathcal{B}_X$ such that
	\begin{align}\label{1}\mathcal{A}_c^{\vec{v}}(b)=L^2(X,\mathcal{K}_\mu^{\vec{v}}(b),\mu).\end{align}
	Directly from \eqref{1}, the $\vec{v}$-directional Kronecker algebra of $(X ,\mathcal{B}_X, \mu, T)$ can be defined by
	$$\mathcal{K}_\mu^{\vec{v}}(b)=\left\{B\in\mathcal{B}_X: \overline{\left\{U_T^{(m,n)}1_B :(m,n)\in \Lambda^{\vec{v}}(b) \right\}}\text{ is compact in } L^2(X,\mathcal{B}_X,\mu) \right\}.$$
	By\cite[Proposition 3.1]{LX}, the definition of $\mathcal{K}_\mu^{\vec{v}}(b)$ is independent of the selection of $b\in (0,\infty)$. So we omit $b$ in $\mathcal{K}_\mu^{\vec{v}}(b)$ and write it as $\mathcal{K}_\mu^{\vec{v}}$.
	We say $\mu$ has $\vec{v}$-discrete spectrum if $\mathcal{K}_\mu^{\vec{v}}=\mathcal{B}_X$.
	
	Motivated by Park \cite{P}, Liu and Xu \cite[Lemma 4.4]{liu2023directional} introduced a skew product to investigate  directional Pinsker algebra. Let $\O=[0,1)$, $\mathcal{C}$
	be the Borel  $\sigma$-algebra on $\O$, and $m$ be the Lebesgue measure on $\O$. As $\O$ is a compact Abelian group, we always use the bi-invariant metric $\rho$ on it. Let $R_\b$ be the rotation on $\O$, i.e., $R_\b t=t+\b\pmod 1$ for any $t\in\O$, where $\b$ is the irrational direction.  Define a skew product system $(\O\times X,\mathcal{C}\times \mathcal{B}_X,m\times \mu, \T)$ by 
	\begin{equation}\label{eq:skew product}
		\T:\O\times X\to \O\times X,\text{ }(\o,x)\mapsto (R_\b t,\varphi(1,\o)x),
	\end{equation}
	where $\varphi(1,\o)=T^{(1,[\b+t])}$, and $[a]$ is the largest natural number less than $a\in\mathbb{R}$. Define a metric on $\O\times X$ by 
	\[\widetilde{d}((t,x),(s,y)):=\max\{\rho(t,s),d(x,y)\}\text{ for any }(t,x),(s,y)\in \O\times X.\]
	For convenience, we denote \[\widetilde{X}=\O\times X,\text{ }\widetilde{\mathcal{B}}=\mathcal{C}\times \mathcal{B}_X,\text{ and }\widetilde{\mu}=m\times \mu.\] 
	Let $\mathcal{K_{\widetilde{\mu}}}$ be the Kronecker algebra of $(\widetilde{X},\widetilde{\mathcal{B}},\widetilde{\mu},\T).$  By a proof similar to that of  \cite[Lemma 4.4, Step 1]{LX}, we have the following result.
	\begin{lem}\label{lem6}
		$\mathcal{K_{\widetilde{\mu}}}=\mathcal{C}\times\mathcal{K}^{\vec{v}}_{\mu}$.
	\end{lem}
	
	Applying Lemma \ref{lem5} and  Lemma \ref{lem6}, we have the following result.
	\begin{lem}\label{lem7}
		Let  $(X_1,\mathcal{B}_{X_1},\mu,T_1)$ and $(X_2,\mathcal{B}_{X_2},\nu,T_2)$ be two $\mathbb{Z}^2$-m.p.s. Then for any direction vector $\vec{v}=(1,\beta)\in\mathbb{R}^2$, one has $$\mathcal{K}^{\vec{v}}_{\mu\times \nu}(X_1\times X_2)=\mathcal{K}^{\vec{v}}_{\mu}(X_1)\times \mathcal{K}^{\vec{v}}_{\nu}(X_2).$$
	\end{lem}
	As a corollary of Lemma \ref{lem7}, the product system of two systems with directional discrete spectrum along the same direction also has directional discrete spectrum along this direction.
	
	\subsection{Weak mixing for $\mathbb{Z}$-actions.} Given a $\mathbb{Z}$-m.p.s. $(X,\mathcal{B}_X,\mu,T)$,  $\mu$ is called weak mixing if $\mathcal{K}_\mu=\{X,\emptyset\}$, and  a function $f\in L^2(X,\mathcal{B}_X,\mu)$ is said to be weakly mixing if $\mathbb{E}(f|\mathcal{K}_{\mu})(x)=0$ for $\mu$-a.e. $x\in X$. Denote by $WM^{\mu}(X)$ the set of weakly mixing functions for $\mu$.
	Then we have the following characterization of weakly mixing systems via weakly mixing functions.
	\begin{prop}
		$\mu$ is  weakly mixing if and only if each $f\in L^2(X,\mathcal{B}_X,\mu)$ with mean zero is weakly mixing. 
	\end{prop}
	Let $S\subset \mathbb{Z}_+$ be an infinite sequence. We define the upper  and lower density of $S$ in $\mathbb{Z}_+$ by
	$$\overline{d}(S)=\limsup\limits_{n\to \infty}\frac{1}{n}\#(S\cap\{0,1,\ldots, n-1\})$$
	and
	$$\underline{d}(S)=\liminf\limits_{n\to \infty}\frac{1}{n}\#(S\cap\{0,1,\ldots, n-1\}).$$ 
	If $\overline{d}(S)=\underline{d}(S)=d$, we say the sequence $S$ has the density $d$.
	
	We recall the Koopman-von Neumann spectrum mixing theorem \cite{KN}.
	\begin{lem}\label{thm9}
		The Hilbert space $\mathcal{H}:=L^2(X,\mathcal{B}_X,\mu)$ can be decomposed as $$\mathcal{H}=L^2(X,\mathcal{K}_{\mu},\mu)\bigoplus WM^{\mu}(X).$$
		Moreover, one has
		\begin{equation*}
			\begin{split}
				W&M^{\mu}(X)\\
				&=\{f\in \mathcal{H}:\exists S\subset \mathbb{N},\text{ }d(S)=1\text{ such that } \text{for all } g\in \mathcal{H},\lim\limits_{n\to \infty\atop n\in S}\langle U_T^nf,g\rangle=0\},
			\end{split}
		\end{equation*}
		where $d(S)$ is the density of $S$ and $ \langle f,g\rangle=\int_X f(x)\overline{g(x)}dx  $ is an inner product on $\mathcal{H}$.
	\end{lem}

	\section{Directional weak mixing}\label{sec:weak mixing}
	In this section, we introduce the notion of directional weak mixing  and provide many characterizations of it.
	\begin{defn}
		Let  $(X,\mathcal{B}_X,\mu,T)$ be a $\mathbb{Z}^2$-m.p.s. and $\vec{v}=(1,\beta)\in\mathbb{R}^2$ be a direction vector. The system $(X,\mathcal{B}_X,\mu,T)$ is $\vec{v}$-weakly mixing if  $\mathcal{K}_\mu^{\vec{v}}=\{X,\emptyset\}$.
	\end{defn}
	\begin{rem}
		The definition of directional weak mixing is independent on the choice of $b\in(0,\infty)$.
	\end{rem}
	With the help of Lemma \ref{lem7}, we obtain the following consequence by the definition of directional weak mixing.
	\begin{prop}
		If two $\mathbb{Z}^2$-m.p.s. $(X_1,\mathcal{B}_{X_1},\mu,T_1)$ and $(X_2,\mathcal{B}_{X_2},\nu,T_2)$ are $\vec{v}$-weakly mixing along a directional vector $\vec{v}=(1,\beta)\in\mathbb{R}^2$,  so does $(X_1\times X_2,\mathcal{B}_{X_1}\times \mathcal{B}_{X_2},\mu\times\nu, T_1\times T_2)$.
	\end{prop}
	
	\subsection{Characterization of directional weak mixing by independence}
	\begin{thm}\label{thm13}
		Let  $(X,\mathcal{B}_X,\mu,T)$ be a $\mathbb{Z}^2$-m.p.s., $\vec{v}=(1,\beta)\in\mathbb{R}^2$ be a direction vector and $b\in(0,\infty)$. Then the following two statements are equivalent.
		\begin{itemize}
			\item [(a)] $(X,\mathcal{B}_X, \mu, T)$ is $\vec{v}$-weakly mixing.
			\item [(b)] For any $B,C\in \mathcal{B}_X$, $$\lim_{k\rightarrow \infty}\frac{1}{\#(\Lambda_k^{\vec{v}}(b))}\sum_{(m,n)\in\Lambda_k^{\vec{v}}(b)}|\langle U_T^{(m,n)}1_B,1_C\rangle-\mu(B)\mu(C)|=0.$$
		\end{itemize}
	\end{thm}
	\begin{proof}
		(b)$\Rightarrow$(a). Suppose to the contradiction that there exists $B\in \mathcal{B}_X$ with $0<\mu(B)<1$ and $B\in \mathcal{K}^{\vec{v}}_\mu$.
		Then  $\overline{\lk\{U_T^{(m,n)}1_B:(m,n)\in \Lambda^{\vec{v}}(b)\re\}}$ is a compact subset of $L^2(X,\mathcal{B}_X,\mu)$.
		Suppose that $\{(m_i,n_i)\}_{i=1}^s$ is an $\epsilon$-net with $\epsilon=\mu(B)-\mu(B)^2$, i.e. for any $(m,n)\in \Lambda^{\vec{v}}(b)$, we have $$\mu\lk(T^{-(m,n)}B\Delta T^{-(m_{j},n_{j})}B\re)=\|U_T^{(m,n)}1_B-U_T^{(m_{j},n_{j})}1_B\|_2<\epsilon$$ for some $j\in \{1,2,\ldots,s\}$.
		Without loss of generality, we may assume that for $(m_1,n_1)$, there exists a subset $A$ of  $\Lambda^{\vec{v}}(b)$ with $\overline{D}_b(A)\geq 1/s>0$ such that 	for any $(m,n)\in A$
		\begin{equation}\label{13}
			\begin{split}
				\epsilon&>\mu\lk(T^{-(m,n)}B\Delta T^{-(m_{1},n_{1})}B\re) \\
				&=2(\mu(B)-\mu\lk(T^{-(m,n)}B\cap T^{-(m_{1},n_{1})}B\re))\\
				&=2(\mu(B)^2-\mu\lk(T^{-(m,n)}B\cap T^{-(m_{1},n_{1})}B\re))+2(\mu(B)-\mu(B)^2),
			\end{split}
		\end{equation}
		which implies that $$\mu(B)^2-\mu\lk(T^{-(m,n)}B\cap T^{-(m_{1},n_{1})}B\re)\overset{\eqref{13}}<\frac{1}{2}\epsilon-(\mu(B)-\mu(B)^2)=-\frac{1}{2}\epsilon.$$
		Let $C=T^{-(m_{1},n_{1})}B$. Then
		\begin{align*}
			&\limsup_{k\rightarrow \infty}\frac{1}{\#(\Lambda_k^{\vec{v}}(b))}\sum_{(m,n)\in\Lambda_k^{\vec{v}}(b)}|\langle U_T^{(m,n)}1_B,1_C\rangle-\mu(B)\mu(C)|\\
			\geq& \limsup_{k\rightarrow \infty}\frac{1}{\#(\Lambda_k^{\vec{v}}(b))}\sum_{(m,n)\in A}|\langle U_T^{(m,n)}1_B,1_C\rangle-\mu(B)\mu(C)|\\
			=&\limsup_{k\rightarrow \infty}\frac{1}{\#(\Lambda_k^{\vec{v}}(b))}\sum_{(m,n)\in A}|\mu(B)^2-\mu\lk(T^{-(m,n)}B\cap T^{-(m_{1},n_{1})}B\re)|\\
			\geq&\limsup_{k\rightarrow \infty}\frac{1}{\#(\Lambda_k^{\vec{v}}(b))}\sum_{(m,n)\in A}\frac{\epsilon}{2}=\frac{\epsilon}{2}\times \overline{D}_b(A)>0,
		\end{align*}
		a contradiction to (b).
		
		(a)$\Rightarrow$(b). Since $(X,\mathcal{B}_X, \mu, T)$ is $\vec{v}$-weakly mixing, we have $\mathcal{K}^{\vec{v}}_\mu=\{\emptyset, X\}$. We recall the skew product $(\widetilde{X},\widetilde{\mathcal{B}},\widetilde{\mu},\T)$ define in \eqref{eq:skew product}. 
		Then by  Lemma \ref{lem6}, we obtain that
		\begin{align*}
			\mathbb{E}(1_{\O}\times 1_{B} |\mathcal{K}_{\widetilde{\mu}})=\mathbb{E}(1_{\O}\times 1_{B}|\mathcal{C}\times\mathcal{K}^{ \vec{v}}_{\mu})=\mathbb{E}(1_{B}|\mathcal{K}^{\vec{v}}_\mu)=\mu(B).
		\end{align*}
		Using Theorem \ref{thm9}, there exists a sequence $\widetilde{S} \subset \mathbb{Z}_+$ with $d(\widetilde{S})=1$ such that for any  $A\in\mathcal{C}$,
		\begin{align}\label{2}
			0=&\lim\limits_{n\to \infty\atop n\in \widetilde{S}}\langle U_{\T}^n ( 1_{\O}\times 1_{B}-\mathbb{E}( 1_{\O}\times 1_{B}|\mathcal{K}_{\widetilde{\mu}})),1_A\times1_{C}\rangle\notag \\
			&=\lim\limits_{n\to \infty\atop n\in \widetilde{S}}\langle U_{\T}^n ( 1_{\O}\times 1_{B})-\mu(B),1_A\times1_{C}\rangle\notag \\
			&=\lim\limits_{n\to \infty\atop n\in \widetilde{S}}\langle 1_{ [0, 1-\{n\beta\})\times T^{-(n,[n\beta])}B}+
			1_{[1-\{n\beta\},1)\times T^{-(n,[n\beta]+1)}B },1_A\times1_{C}\rangle\notag \\
			&\quad-\mu(B)\mu(C)m(A)\notag\\
			&=\lim\limits_{n\to \infty\atop n\in \widetilde{S}}(\mu(T^{-(n,[n\beta])}B\cap C)m([0,1-\{n\beta\})\cap A)\notag\\
			&\quad+\mu(T^{-(n,[n\beta]+1)}B\cap C)m([1-\{n\beta\},1)\cap A))\notag\\
			&\quad-\mu(B)\mu(C)m(A).
		\end{align}
		Now we prove 
		\begin{equation}\label{eq:ave}
			\lim\limits_{k\to \infty}\frac{1}{k}\sum_{n=0}^{k-1}|\mu(T^{-(n,[n\beta])}B\cap C)-\mu(B)\mu(C)|=0.
		\end{equation}
		By contradiction, 	we assume that there exists  $\epsilon>0$ such that $$\limsup\limits_{k\to \infty}\frac{1}{k}\sum_{n=0}^{k-1}|\mu(T^{-(n,[n\beta])}B\cap C)-\mu(B)\mu(C)|=2\epsilon>0.$$  So there exists a subsequence $P$ of  $\widetilde{S}$ with $\overline{d}(P)>\epsilon$ such that for any $n\in P$, \begin{equation}\label{eq:contr}
			|\mu(T^{-(n,[n\beta])}B\cap C)-\mu(B)\mu(C)|>\epsilon.
		\end{equation}
		Applying pointwise ergodic theorem on the ergodic $\mathbb{Z}$-m.p.s. $(\O,\mathcal{C},m,R_\b)$, one has that
		$$\lim\limits_{k\to \infty}\frac{1}{k}\#\{n\in [0,k-1]: 1-\{n\beta\}\in [0,\frac{\epsilon }{2})\}=\frac{\epsilon}{2}.$$
		Then there exists a subsequence $Q$ of $P$ with $\overline{d}(Q)>\frac{\epsilon}{2}$
		such that $|\mu(T^{-(n,[n\beta])}B\cap C)-\mu(B)\mu(C)|>\epsilon$ and $1-\{n\beta\}\in [\frac{\epsilon }{2},1)$ for any $n\in Q$.
		Substitute $A=[0,\frac{\epsilon }{2})$ into \eqref{2}, we obtain that
		\begin{align*}
			0&=\lim\limits_{n\to \infty,n\in Q}
			\mu(T^{-(n,[n\beta])}B\cap C)m( A)-\mu(B)\mu(C)m(A)\\
			&=\frac{\epsilon }{2}\cdot\lim\limits_{n\to \infty,n\in Q}(\mu\lk(T^{-(n,[n\beta])}B\cap C)-\mu(B)\mu(C)\re),
		\end{align*}
		which implies that $\lim\limits_{n\to \infty,n\in Q}|\mu(T^{-(n,[n\beta])}B\cap C)-\mu(B)\mu(C)|=0$, a contradiction to \eqref{eq:contr}. Thus, \eqref{eq:ave} holds, which together with the invariance of $\mu$, implies that 	for any $l\in \mathbb{Z}$,
		\begin{equation}\label{14}
			\lim\limits_{k\to \infty}\frac{1}{k}\sum_{n=0}^{k-1}|\mu(T^{-(n,[n\beta]+l)}B\cap C)-\mu(B)\mu(C)|=0.
		\end{equation}
		
		Now, we finish the proof of (b), which is divided into two cases.
		\item[\textbf{Case 1}.]
		When $b\ge1/2$, it is easy to see that $\#(\Lambda_k^{\vec{v}}(b))\ge 2k.$
		Hence
		\begin{align*}
			&\lim_{k\rightarrow \infty}\frac{1}{\#(\Lambda_k^{\vec{v}}(b))}\sum_{(m,n)\in\Lambda_k^{\vec{v}}(b)}|\langle U_T^{(m,n)}1_B,1_C\rangle-\mu(B)\mu(C)|\\
			\le& \lim_{k\rightarrow \infty}\frac{1}{k}\cdot
			\sum_{l=-([b]+1)}^{[b]+1}\sum_{n=0}^{k-1}|\mu(T^{-(n,[n\beta]+l)}B\cap C)-\mu(B)\mu(C)|\\
			=&  \sum_{l=-([b]+1)}^{[b]+1}\lim_{k\rightarrow \infty}\frac{1}{k}\cdot\sum_{n=0}^{k-1}|\mu(T^{-(n,[n\beta]+l)}B\cap C)-\mu(B)\mu(C)|\\
			\overset{\eqref{14}}=&0.
		\end{align*}

		\item[\textbf{Case 2.}]
		When $ 0<b< 1/2$,
		for each $k\in\mathbb{N}$, $\#(\Lambda_k^{\vec{v}}(b))=
		\#(\mathcal{F}_k)$, where $\mathcal{F}_k=\{0\le m\le k-1: \{m\beta\} \in [0,b]\cup [1-b,1)\}$.
		Using pointwise ergodic theorem again, one has
		$
		\lim\limits_{k\to \infty}\frac{1}{k}\#(\mathcal{F}_k)=b.
		$
		Therefore,
		\begin{align*}
			&\lim_{k\rightarrow \infty}\frac{1}{\#(\Lambda_k^{\vec{v}}(b))}\sum_{(m,n)\in\Lambda_k^{\vec{v}}(b)}|\langle U_T^{(m,n)}1_B,1_C\rangle-\mu(B)\mu(C)|\\
			\le& \lim_{k\rightarrow \infty}\frac{k}{\#(\mathcal{F}_k)}\cdot \frac{1}{k}\sum_{n=0}^{k-1}|\mu(T^{-(n,[n\beta]+l)}B\cap C)-\mu(B)\mu(C)|\\
			=&\frac{1}{2b}\cdot0\overset{\eqref{14}}=0.
		\end{align*}
		The conclusion in (b) is obtained by two cases above-mentioned.
	\end{proof}
	
	\begin{cor}\label{cor3}
		Let  $(X,\mathcal{B}_X,\mu,T)$ be a $\mathbb{Z}^2$-m.p.s., $\vec{v}=(1,\beta)\in\mathbb{R}^2$ be a direction vector and $b\in (0,\infty)$. Then the following statements are equivalent.
		\begin{itemize}
			\item [(a)] $(X,\mathcal{B}_X, \mu, T)$ is $\vec{v}$-weakly mixing.
			\item [(b)]  For any $B,C\in\mathcal{B}_X$, there exists a subset $\mathcal{Q}_{B,C}$ of $\Lambda_+^{\vec{v}}(b)$ with
			$D_b(\mathcal{Q}_{B,C})=1$ such that for any subset $\{(m_i,n_i)\}_{i=1}^\infty$ of $\mathcal{Q}_{B,C}$ that satisfies $\{m_i\}_{i=1}^\infty$ is a strictly monotone increasing sequence, one has 
			\[\lim_{i\to\infty}\mu(T^{-(m_i,n_i)}B\cap C)=\mu(B)\mu(C).\]
			\item [(c)] 
			For any $B,C\in\mathcal{B}_X$, there exists a subset $\mathcal{Q}_{B,C}$ of $\Lambda_+^{\vec{v}}(b)$ with
			$D_b(\mathcal{Q}_{B,C})=1$ such that for any $\{(m_i,n_i)\}_{i=1}^\infty$, one has 
			\[\lim_{i\to\infty}\mu(T^{-(m_i,n_i)}B\cap C)=\mu(B)\mu(C).\]
		\end{itemize}
	\end{cor}
	\begin{proof}
		(a)$\Rightarrow$(b). For each $p\in\mathbb{N}$,  let 
		\[\mathcal{A}_{B,C}^{p}:=\{(m,n)\in\Lambda^{\vec{v}}(b):|\mu(T^{-(m,n)}B\cap C)-\mu(B)\mu(C)|\geq 1/p\}.\]
		Then 
		\[(1/p)\cdot\frac{\#(\mathcal{A}_{B,C}^{p}\cap([0,k-1]\times\mathbb{Z}))}{\#(\Lambda_k^{\vec{v}}(b))}\leq \frac{1}{\#(\Lambda_k^{\vec{v}}(b))}\sum_{(m,n)\in\Lambda_k^{\vec{v}}(b)}|\mu(T^{-(m,n)}B\cap C)-\mu(B)\mu(C)|, \]
		which together with Theorem \ref{thm13}, implies that
		\[\lim_{k\to \infty}\frac{\#(\mathcal{A}_{B,C}^{p})\cap([0,k-1]\times\mathbb{Z})}{\#(\Lambda_k^{\vec{v}}(b))}=0.\]
		Then for each $p\in\mathbb{N}$, there exists $l_p\in\mathbb{N}$ such that for any $k\geq l_p$, one has 
		\[\frac{\#(\mathcal{A}_{B,C}^{p})\cap([0,k-1]\times\mathbb{Z})}{\#(\Lambda_k^{\vec{v}}(b))}<\frac{1}{p+1}.\]
		Let \[\mathcal{A}_{B,C}:=\bigcup_{p=0}^\infty \lk(\mathcal{A}_{B,C}^{p}\cap ([l_p,l_{p+1}-1]\times \mathbb{Z})\re),\]
		where $l_0=0$.
		If $l_p\leq k\leq l_{p+1}-1$, then 
		\[\mathcal{A}_{B,C}\cap ([0,k-1]\times\mathbb{Z})=\lk(\mathcal{A}_{B,C}\cap ([0,l_p-1]\times\mathbb{Z})\re)\cup \lk(\mathcal{A}_{B,C}\cap ([l_p,k-1]\times\mathbb{Z})\re).\]
		Hence  for each $p\in\mathbb{N}$,
		\begin{equation*}
			\begin{split}
				&\frac{\#(\mathcal{A}_{B,C})\cap([0,k-1]\times\mathbb{Z})}{\#(\Lambda_k^{\vec{v}}(b))}\\
				=&\frac{\#\lk(\mathcal{A}_{B,C}\cap ([0,l_p-1]\times\mathbb{Z})\re)+ \#\lk(\mathcal{A}_{B,C}\cap ([l_p,k-1]\times\mathbb{Z})\re)}{\#(\Lambda_k^{\vec{v}}(b))}\\
				\leq &\frac{\#\lk(\mathcal{A}_{B,C}^{p-1}\cap ([0,l_p-1]\times\mathbb{Z})\re)+ \#\lk(\mathcal{A}_{B,C}^{p}\cap ([l_p,k-1]\times\mathbb{Z})\re)}{\#(\Lambda_k^{\vec{v}}(b))}\\
				\leq &\frac{1}{p}+\frac{ \#\lk(\mathcal{A}^p_{B,C}\cap ([0,k-1]\times\mathbb{Z})\re)}{\#(\Lambda_k^{\vec{v}}(b))}\leq \frac{1}{p}+\frac{1}{p+1}.
			\end{split}
		\end{equation*}
		Let   $k\to \infty$. Then
		\[\lim_{k\to \infty}\frac{\#(\mathcal{A}_{B,C}\cap([0,k-1]\times\mathbb{Z}))}{\#(\Lambda_k^{\vec{v}}(b))}=0.\]
		Let $\mathcal{Q}_{B,C}:=\Lambda_k^{\vec{v}}(b)\setminus\mathcal{A}_{B,C}.$ Then 
		\[D_b(\mathcal{Q}_{B,C})=\lim_{k\to \infty}\frac{\#(\mathcal{Q}_{B,C}\cap([0,k-1]\times\mathbb{Z}))}{\#(\Lambda_k^{\vec{v}}(b))}=1.\]
		Moreover, for any infinite subset $\{(m_i,n_i)\}_{i=1}^\infty$ of $\mathcal{Q}_{B,C}$ that satisfies $\{m_i\}_{i=1}^\infty$ is a strictly monotone increasing sequence, if $m_i>l_p$ then $(m_i,n_i)\notin \mathcal{A}_{B,C}^p$, that is, 
		$|\mu(T^{-(m_i,n_i)}B\cap C)-\mu(B)\mu(C)|<1/p.$
		Thus, one has
		\[\lim_{i\to\infty}\mu(T^{-(m_i,n_i)}B\cap C)=\mu(B)\mu(C).\]
		
		(b)$\Rightarrow$(c). Given $B,C\in\mathcal{B}_X$,   let $\{(m_i,n_i)\}_{i=1}^\infty$ be an infinite sequence of $\mathcal{Q}_{B,C}$. Since $b\in(0,\infty)$ is finite, there exist $\{(m^{(j)}_i,n^{(j)}_i)\}_{i=1}^\infty$ that satisfy $\{m^{(j)}_i\}_{i=1}^\infty$ is a strictly monotone increasing sequence with respect to $i$ for each  $j=1,2,\ldots,r$ and $\{(s_i,t_i)\}_{i=1}^q$ for some $q\in\mathbb{N}$ such that
		\begin{itemize}
			\item 	$\{(m^{(1)}_i,n^{(1)}_i)\}_{i=1}^\infty,\ldots,\{(m^{(r)}_i,n^{(r)}_i)\}_{i=1}^\infty$ and $\{(s_i,t_i)\}_{i=1}^q$ are mutually disjoint.
			\item $\{(m_i,n_i)\}_{i=1}^\infty=\bigcup_{j=1}^r\{(m^{(j)}_i,n^{(j)}_i)\}_{i=1}^\infty\bigcup\{(s_i,t_i)\}_{i=1}^q$.
		\end{itemize} 
		By (b), we obtain that 
		$\lim_{i\to\infty}\mu(T^{-(m^{(j)}_i,n^{(j)}_i)}B\cap C)=\mu(B)\mu(C),$
		for each $j=1,2,\ldots,r$. This implies that \[\lim_{i\to\infty}\mu(T^{-(m_i,n_i)}B\cap C)=\mu(B)\mu(C).\]
		
		(c)$\Rightarrow$(a). By (c), for any $B,C\in\mathcal{B}_X$, there exists a subset $\mathcal{Q}_{B,C}$ of $\Lambda^{\vec{v}}(b)$ with
		$D_b(\mathcal{Q}_{B,C})=1$, so that for any $\{(s_j,t_j)\}_{j=1}^\infty$, one has 
		\[\lim_{j\to\infty}\mu(T^{-(s_j,t_j)}B\cap C)=\mu(B)\mu(C).\]
		In particular,  if we enumerate $\mathcal{Q}_{B,C}=\{(m_i,n_i)\}_{i=1}^\infty$, then for any $\epsilon>0$, there exists $I>0$ such that $i\geq I$ implies that 
		\begin{equation}\label{666}
			|\mu(T^{-(m_i,n_i)}B\cap C)-\mu(B)\mu(C)|<\epsilon/2.
		\end{equation}
		Since 	$D_b(\mathcal{Q}_{B,C})=1$, there exists $K>I$ such that for any $k>K$, one has 
		\begin{equation}\label{555}
			\frac{\#(\Lambda_k^{\vec{v}}(b)\setminus\mathcal{Q}_{B,C})}{\#(\Lambda_k^{\vec{v}}(b))}<\epsilon/4.
		\end{equation}
		Hence for any $k>K$,
		\begin{equation*}
			\begin{split}
				&\frac{1}{\#(\Lambda_k^{\vec{v}}(b))}\sum_{(m,n)\in\Lambda_k^{\vec{v}}(b)}|\mu(T^{-(m,n)}B\cap C)-\mu(B)\mu(C)|\\
				\leq& \frac{1}{\#(\Lambda_k^{\vec{v}}(b))}\sum_{(m,n)\in\mathcal{Q}_{B,C}\cap([0,k-1]\times \mathbb{Z})}|\mu(T^{-(m,n)}B\cap C)-\mu(B)\mu(C)|+\frac{2\#(\Lambda_k^{\vec{v}}(b)\setminus\mathcal{Q}_{B,C})}{\#(\Lambda_k^{\vec{v}}(b))}\\
				&\overset{\eqref{666},\eqref{555}}\leq  (\epsilon/2)\cdot\frac{\#(\mathcal{Q}_{B,C}\cap([0,k-1]\times \mathbb{Z}))}{\#(\Lambda_k^{\vec{v}}(b))}+\epsilon/2<\epsilon.
			\end{split}
		\end{equation*}
		Therefore
		\[\lim_{k\to \infty}\frac{1}{\#(\Lambda_k^{\vec{v}}(b))}\sum_{(m,n)\in\Lambda_k^{\vec{v}}(b)}|\mu(T^{-(m,n)}B\cap C)-\mu(B)\mu(C)|=0,\]
		which implies that $(X,\mathcal{B}_X, \mu, T)$ is $\vec{v}$-weakly mixing by Theorem \ref{thm13}.
	\end{proof}
	
	With the help of Theorem \ref{thm13}, by a proof similar to that of $\mathbb{Z}$-m.p.s., we have the following theorem.
	\begin{thm}Let  $(X,\mathcal{B}_X,\mu,T)$ be a $\mathbb{Z}^2$-m.p.s., $\vec{v}=(1,\beta)\in\mathbb{R}^2$ be a direction vector and $b\in (0,\infty)$. Then the following statements are equivalent.
		\begin{itemize}
			\item [(a)] $(X,\mathcal{B}_X, \mu, T)$ is $\vec{v}$-weakly mixing.
			\item [(b)]  For any $f\in L^2(X,\mathcal{B}_X,\mu)$, $$\lim_{k\rightarrow \infty}\frac{1}{\#(\Lambda_k^{\vec{v}}(b))}\sum_{(m,n)\in\Lambda_k^{\vec{v}}(b)}|\langle U_T^{(m,n)}f,f\rangle-\langle f,1\rangle\langle 1,f\rangle|=0.$$
			\item [(c)] 
			For any $f,g\in L^2(X,\mathcal{B}_X,\mu)$, $$\lim_{k\rightarrow \infty}\frac{1}{\#(\Lambda_k^{\vec{v}}(b))}\sum_{(m,n)\in\Lambda_k^{\vec{v}}(b)}|\langle U_T^{(m,n)}f,g\rangle-\langle f,1\rangle\langle1,g\rangle|=0.$$
		\end{itemize}
	\end{thm}
	
	\subsection{Characterization of directional weak mixing by directional sequence entropy}
	In this subsection, we describe directional weak mixing via directional sequence entropy. For this purpose, we need a consequence \cite[Theorem 1.1]{LX}, which is restated as follows.
	\begin{lem} \label{thm-1}
		Let  $(X,\mathcal{B}_X,\mu,T)$ be a $\mathbb{Z}^2$-m.p.s.,  $\vec{v}=(1,\beta)\in\mathbb{R}^2$ be a direction vector and $b\in (0,\infty)$. Given a finite measurable partition $\alpha$ of $X$, for any infinite subset $S'$ of $\Lambda^{\vec{v}}(b)$,
		$$h^{S'}_{\mu}(T,\alpha)\leq H_\mu(\alpha|\mathcal{K}_\mu^{\vec{v}}).$$
		Moreover, there exists an infinite sequence $S=\{(m_i,n_i)\}_{i=1}^{\infty}$ of $ \Lambda^{\vec{v}}(b)$ such that $\{m_i\}_{i=1}^{\infty}$ is strictly monotone and $$h^{S}_{\mu}(T,\alpha)=H_\mu(\alpha|\mathcal{K}_\mu^{\vec{v}}).$$
	\end{lem}
	With the help of  the above lemma, we have the following description.
	\begin{thm}\label{thm-2}
		Let  $(X,\mathcal{B}_X,\mu,T)$ be a $\mathbb{Z}^2$-m.p.s., $\vec{v}=(1,\beta)\in\mathbb{R}^2$ be a  vector and $b\in(0,\infty)$. Then the following three conditions are equivalent.
		\begin{itemize}
			\item[(a)]	$(X,\mathcal{B}_X,\mu,T)$ is  $\vec{v}$-weakly mixing.	
			\item[(b)] For any  $ B\in\mathcal{B}_X$ with $0<\mu(B)<1$, there exists an infinite  subset $S=\{(m_i,n_i)\}_{i=1}^{\infty}$ of $\Lambda^{\vec{v}}(b)$ such that  $h^{S}_{\mu}(T,\{B,B^c\})>0.$
			\item[(c)]For any nontrivial finite measurable partition $\alpha$ of $X$, there exists an infinite subset $S=\{(m_i,n_i)\}_{i=1}^{\infty}$ of $\Lambda^{\vec{v}}(b)$ such that  $h^{S}_{\mu}(T,\alpha)>0.$
		\end{itemize}
		\begin{proof}
			(a) $\Rightarrow$ (b).  Given $ B\in\mathcal{B}_X$ with $0<\mu(B)<1$,
			applying Lemma \ref{thm-1} on the measurable partition $\{B,B^c\}$, there exists an infinite sequence $S=\{(m_i,n_i)\}_{i=1}^{\infty}$ of $ \Lambda^{\vec{v}}(b)$ such that  $$h^{S}_{\mu}(T,\{B,B^c\})=H_\mu(\{B,B^c\}|\mathcal{K}_\mu^{\vec{v}}).$$
			Since $(X,\mathcal{B}_X,\mu,T)$ is $\vec{v}$-weakly mixing, one has $B\notin \{\emptyset,X\}=\mathcal{K}_\mu^{\vec{v}}$, and hence $$h^{S}_{\mu}(T,\{B,B^c\})=H_\mu(\{B,B^c\}|\mathcal{K}_\mu^{\vec{v}})>0.$$
			
			(b) $\Rightarrow$ (c). For any  $B\in \alpha$ with $0< \mu(B)<1$, by (b), there exists an infinite  sequence $S=\{(m_i,n_i)\}_{i=1}^{\infty}$ of $\Lambda^{\vec{v}}(b)$ such that $h^{S}_{\mu}(T,\{B,B^c\})>0.$
			Since $\alpha$ is finer than $\{B,B^c\}$, it follows that $$h^{S}_{\mu}(T,\alpha)\geq h^{S}_{\mu}(T,\{B,B^c\})>0.$$
			
			(c) $\Rightarrow$ (a). By contradiction,  we assume that $(X,\mathcal{B}_X,\mu,T)$ is not $\vec{v}$-weakly mixing. Then there exists $B\in \mathcal{K}_\mu^{\vec{v}}$ with $0<\mu(B)<1$ such that, by Lemma \ref{thm-1}, for any infinite subset $S'$ of $\Lambda^{\vec{v}}(b)$,
			$$h^{S'}_{\mu}(T,\{B,B^c\})\leq H_\mu(\{B,B^c\}|\mathcal{K}_\mu^{\vec{v}})=0,$$
			a contradiction to (c). Therefore $(X,\mathcal{B}_X,\mu,T)$ is $\vec{v}$-weakly mixing.
		\end{proof}
	\end{thm}
	The following theorem demonstrates that the sequence $S$ in Theorem \ref{thm-2} can be consistent with respect to all finite measurable partitions.
	\begin{thm}
		Let  $(X,\mathcal{B}_X,\mu,T)$ be a $\mathbb{Z}^2$-m.p.s., $\vec{v}=(1,\beta)\in\mathbb{R}^2$ be a directional vector and $b\in(0,\infty)$. Then the following two conditions are equivalent.
		\begin{itemize}
			\item[(a)]	$(X,\mathcal{B}_X,\mu,T)$ is  $\vec{v}$-weakly mixing.	
			\item[(b)]There exists an infinite subset $S=\{(m_i,n_i)\}_{i=1}^{\infty}$ of $\Lambda^{\vec{v}}(b)$ such that  $h^{S}_{\mu}(T,\alpha)=H_{\mu}(\alpha)$
			for any finite measurable partition $\alpha$ of $X$.
		\end{itemize}
	\end{thm}
	\begin{proof}
		(a) $\Rightarrow$ (b). Since $(X,\mathcal{B}_X,\mu,T)$ is $\vec{v}$-weakly mixing, by Corollary \ref{cor3}, one has  for any $B,C\in\mathcal{B}_X$, there exists a subsequence $\mathcal{Q}_{B,C}$ of $\Lambda^{\vec{v}}(b)$ with
		$D_b(\mathcal{Q}_{B,C})=1$ such that for any infinite sequence $\{(m_i,n_i)\}_{i=1}^\infty$ of $\mathcal{Q}_{B,C}$,
		\begin{equation}\label{33}
			\lim_{i\to\infty}\mu(T^{-(m_i,n_i)}B\cap C)=\mu(B)\mu(C).
		\end{equation}
		Since $\mathcal{B}_X$ is separable, there exists a countably  subsets $\{\eta_l\}_{l=1}^\infty$ of finite partitions such that  $\eta_l\uparrow\mathcal{B}_X$.
		Now we inductively define  an infinite sequence $S=\{(m_i,n_i)\}_{i=1}^\infty$  such that 
		\begin{equation}\label{15}
			H_{\mu}\lk(T^{-(m_j,n_j)}\eta_l|\bigvee_{i=1}^{j-1}T^{-(m_i,n_i)}\eta_l\re)\geq H_\mu(\eta_l)-2^{-j}
		\end{equation}
		for all $j\geq 2$ and $1\leq l\leq j.$ Suppose $(m_1,n_1),\ldots,(m_{j-1},n_{j-1})$ have been defined. Since $\eta_l$ is finite for each $l\in\mathbb{N}$,  $\mathcal{Q}_{j}:=\cap_{l=1}^j\cap_{B\in\eta_l}\cap_{C\in\bigvee_{i=1}^{j-1}T^{-(m_i,n_i)}\eta_l}\mathcal{Q}_{B,C}$
		is an infinite sequence of $\Lambda^{\vec{v}}(b)$.
		Using \eqref{33}, we may choose $(m_j,n_j)\in \mathcal{Q}_{j}$ such that 
		\begin{equation}\label{44}
			|\log{\mu(T^{-(m_j,n_j)}B\cap C)}-\log{\mu(B)\mu(C)}|<2^{-j}
		\end{equation}
		for all $B\in\eta_l$ and $C\in\bigcup_{l=1}^j\bigvee_{i=1}^{j-1}T^{-(m_i,n_i)}\eta_l$.
		Then
		\begin{align*}
			&H_{\mu}\lk(T^{-(m_j,n_j)}\eta_l|\bigvee_{i=1}^{j-1}T^{-(m_i,n_i)}\eta_l\re)\\
			=&\sum_{B,C}-\mu(T^{-(m_j,n_j)}B\cap C)\log{\frac{\mu(T^{-(m_j,n_j)}B\cap C)}{\mu(C)}}\\
			\overset{\eqref{44}}\geq&\sum_{B,C}\mu(T^{-(m_j,n_j)}B\cap C)\lk(\log{\mu(C)}-\log{({\mu(B)\mu(C)})}-2^{-j}\re)\\
			=&\sum_{B,C}-\mu(T^{-(m_j,n_j)}B\cap C)\log{\mu(B)}-2^{-j}\\
			=&\sum_{B}-\mu(B)\log{\mu(B)}-2^{-j}=H_{\mu}(\eta_l)-2^{-j}.
		\end{align*}
		Thus, we obtain a sequence $\{(m_i,n_i)\}_{i=1}^\infty$ satisfying \eqref{15}.
		
		Given any $l\in\mathbb{N}$, for $j>l$,
		\begin{align*}
			&H_{\mu}\lk(\bigvee_{i=1}^{j}T^{-(m_i,n_i)}\eta_l\re)\\
			=&H_{\mu}\lk(\bigvee_{i=1}^{l}T^{-(m_i,n_i)}\eta_l\re)+\sum_{i=l+1}^{j}H_{\mu}\lk(T^{-(m_i,n_i)}\eta_l|\bigvee_{s=1}^{i-1}T^{-(m_s,n_s)}\eta_l\re)\\
			\overset{\eqref{15}}\geq&H_{\mu}\lk(\bigvee_{i=1}^{l}T^{-(m_i,n_i)}\eta_l\re)+(j-l)H_{\mu}(\eta_l)-\sum_{i=l+1}^{j}2^{-i}.
		\end{align*}
		Therefore, we deduce that
		\begin{equation}\label{55}
			\begin{split}
				h_{\mu}^S(T,\eta_l)=&\limsup_{j\to\infty}\frac{1}{j}H_{\mu}\lk(\bigvee_{i=1}^{j}T^{-(m_i,n_i)}\eta_l\re)\\
				\geq &\limsup_{j\to\infty}\lk(\frac{1}{j}H_{\mu}\lk(\bigvee_{i=1}^{l}T^{-(m_i,n_i)}\eta_l\re)+\frac{j-l}{j}H_{\mu}(\eta_l)-\frac{1}{j}\sum_{i=l+1}^{j}2^{-i}\re)\\
				=&H_{\mu}(\eta_l).
			\end{split}
		\end{equation}
		Now let $\alpha$ be any finite measurable partition of $X$. Given $\delta>0$, by Lemma \ref{lem-2} and the density of $\{\eta_l\}_{l=1}^\infty$, there exists  $l\in\mathbb{N}$ such that 
		\[H_\mu(\alpha|\eta_l)+H_\mu(\eta_l|\alpha)<\delta\text{ and }H_{\mu}(\eta_l)\geq H_{\mu}(\alpha).\]
		This implies that
		\begin{equation*}
			\begin{split}
				h^S_{\mu}(T,\alpha)\geq h^S_{\mu}(T,\eta_l)-H_{\mu}(\eta_l|\alpha)
				\geq  h^S_{\mu}(T,\eta_l)-\delta
				\overset{\eqref{55}}\geq H_{\mu}(\eta_l)-\delta
				\geq  H_{\mu}(\alpha)-\delta.
			\end{split}
		\end{equation*} 
		As $\delta>0$ is arbitrary,  one has $h^S_{\mu}(T,\alpha)\geq H_{\mu}(\alpha)$, which together with Lemma \ref{thm-1},  implies that
		$h^S_{\mu}(T,\alpha)=H_{\mu}(\alpha)$.
		
		(b) $\Rightarrow$ (a). This direction is a direct corollary of Theorem \ref{thm-2}.
	\end{proof}
	
	\subsection{Characterization of directional weak mixing by weakly mixing functions}
	In this section, we use  weakly mixing functions to characterize directional weak mixing. Moreover, we investigate many properties of it.
	
	The following consequence can be easily proved by the definition.
	\begin{prop}\label{prop2}
		Let  $(X,\mathcal{B}_X,\mu,T)$ be a $\mathbb{Z}^2$-m.p.s. and $\vec{v}=(1,\beta)\in\mathbb{R}^2$ be a direction vector. Then the following two statements are equivalent.
		\begin{itemize}
			\item[(a)] $(X,\mathcal{B}_X,\mu,T)$ is $\vec{v}$-weakly mixing.
			\item[(b)] Each $f\in L^2(X,\mathcal{B}_X,\mu)$ with $\int_{X}fd\mu=0$ is $\vec{v}$-weakly mixing.
		\end{itemize}
	\end{prop}
	Now, we provide a directional version of Koopman-von Neumann spectrum mixing theorem. 
	\begin{thm}\label{thm8}
		Let  $(X,\mathcal{B}_X,\mu,T)$ be a $\mathbb{Z}^2$-m.p.s. and $\vec{v}=(1,\beta)\in\mathbb{R}^2$ be a direction vector. Then 
		$$L^2(X,\mathcal{B}_X,\mu)= WM^{\mu,\vec{v}}(X)\bigoplus L^2(X,\mathcal{K}_\mu^{\vec{v}},\mu).$$ 
		\begin{proof}
			Firstly, we show that $L^2(X,\mathcal{K}_\mu^{\vec{v}},\mu)$ is orthogonal to  $WM^{\mu,\vec{v}}(X)$. Indeed, for any $g\in WM^{\mu,\vec{v}}(X)$, we have $\mathbb{E}(g|\mathcal{K}_\mu^{\vec{v}})(x)=0$ for $\mu$-a.e. $x\in X$. Thus, $\mathbb{E}(1_\O\times g|\mathcal{K}_{\widetilde{\mu}})(t,x)=0$ for $\widetilde{\mu}$-a.e. $(t,x)\in \widetilde{X}$, where $\mathcal{K}_{\widetilde{\mu}}$ is the Kronecker algebra of the skew product $(\widetilde{X},\widetilde{\mathcal{B}},\widetilde{\mu},\T)$ defined in \eqref{eq:skew product}.
			So $1_\O\times g$ is a weakly mixing function on $\widetilde{X}$. Similarly,  for any $f\in L^2(X,\mathcal{K}_\mu^{\vec{v}},\mu)$,  $1_\O\times f\in L^2(\widetilde{X},\mathcal{K}_{\widetilde{\mu}},\widetilde{\mu}).$  Applying Lemma \ref{thm9} on the $\mathbb{Z}$-m.p.s. $(\widetilde{X},\mathcal{K}_{\widetilde{\mu}},\widetilde{\mu},\T)$, we conclude $\langle 1_\O\times f,1_\O\times g\rangle_{\widetilde{X}}=0,$ which implies that $\langle f,g \rangle_X=0.$
			
			Next we claim that for any $g\in L^2(X,\mathcal{B}_X,\mu)$ if it satisfies $\langle f,g \rangle_X=0$ for all $f\in L^2(X,\mathcal{K}_\mu^{\vec{v}},\mu)$, then $g\in WM^{\mu,\vec{v}}(X)$.
			By contradiction, we assume that $g\notin WM^{\mu,\vec{v}}(X)$. Let $E=\{x\in X:\mathbb{E}(g|\mathcal{K}_\mu^{\vec{v}})(x)>0\}$. Then $E$ is $\mathcal{K}_\mu^{\vec{v}}$-measurable and either $\mu(E)>0$ or $\mu(E^c)>0$. Without loss of generality, we assume that $\mu(E)>0$. Thus, $$0=\langle 1_E,g \rangle_X=\int_{E}g d\mu=\int_{E}\mathbb{E}(g|\mathcal{K}_\mu^{\vec{v}}) d\mu>0,$$
			a contradiction. Therefore $g\in WM^{\mu,\vec{v}}(X).$
			
			Since $\mathbb{E}(\cdot|\mathcal{K}_\mu^{\vec{v}})$ is the orthogonal projection from $L^2(X,\mathcal{B}_X,\mu)$ to $L^2(X,\mathcal{K}_\mu^{\vec{v}},\mu)$, one has  $\langle f-\mathbb{E}(f|\mathcal{K}_\mu^{\vec{v}}),g \rangle_X=0$ for any $g\in L^2(X,\mathcal{K}_\mu^{\vec{v}},\mu)$. By the above claim, we have  $f-\mathbb{E}(f|\mathcal{K}_\mu^{\vec{v}})\in WM^{\mu,\vec{v}}(X).$
			
			To sum up the above proof,  we have $L^2(X,\mathcal{B}_X,\mu)= WM^{\mu,\vec{v}}(X)\bigoplus L^2(X,\mathcal{K}_\mu^{\vec{v}},\mu).$
		\end{proof}
	\end{thm}
	Now we study the structure of  $WM^{\mu,\vec{v}}(X)$ by asymptotic independence.
	\begin{thm}\label{3.10}
		Let  $(X,\mathcal{B}_X,\mu,T)$ be a $\mathbb{Z}^2$-m.p.s., $\vec{v}=(1,\beta)\in\mathbb{R}^2$ be a direction vector and $b\in(0,\infty)$. Then
		\begin{equation*}
			\begin{split}
				&WM^{\mu,\vec{v}}(X)=\{f\in L^2(X,\mathcal{B}_X,\mu):\\
				&\qquad\qquad \text{ for any }g\in L^2(X,\mathcal{B}_X,\mu), \lim_{k\to \infty}\frac{1}{\#(\Lambda_k^{\vec{v}}(b))}\sum_{(m,n)\in\Lambda_k^{\vec{v}}(b)}|\langle U_T^{(m,n)}f,g\rangle|=0\}.
			\end{split}
		\end{equation*}
	\end{thm}
	\begin{proof}
		Let \begin{equation*}
			\begin{split}
				&\mathfrak{A}:=\{f\in L^2(X,\mathcal{B}_X,\mu):\\
				&\qquad\qquad \text{ for any }g\in L^2(X,\mathcal{B}_X,\mu), \lim_{k\to \infty}\frac{1}{\#(\Lambda_k^{\vec{v}}(b))}\sum_{(m,n)\in\Lambda_k^{\vec{v}}(b)}|\langle U_T^{(m,n)}f,g\rangle|=0\}.
			\end{split}
		\end{equation*}
		
		We firstly prove that $WM^{\mu,\vec{v}}(X)\subset \mathfrak{A}.$
		Given $f\in WM^{\mu,\vec{v}}(X)$, $\mathbb{E}(f|\mathcal{K}_\mu^{\vec{v}})(x)=0$ for $\mu$-a.e. $x\in X$. Note that for any $g\in L^2(X,\mathcal{B}_X,\mu)$, $1_\O\times g\in L^2(\widetilde{X},\widetilde{\mathcal{B}},\widetilde{\mu})$, where $(\widetilde{X},\widetilde{\mathcal{B}},\widetilde{\mu},\T)$ is the skew product system defined  by \eqref{eq:skew product}. Hence 
		\[\mathbb{E}(1_\O\times f|\mathcal{K}_{\mathcal{C}\times\widetilde{\mu}} )=\mathbb{E}(f|\mathcal{K}_\mu^{\vec{v}})\cdot\mathbb{E}(1_\O|\mathcal{C})=0.\]
		By Lemma \ref{thm9} for any $g\in L^2(X,\mathcal{B}_X,\mu)$, there exists a subsequence $S$ of $\mathbb{Z}_+$ with $d(S)=1$ such that 
		\[\lim_{n\to \infty,n\in {S}}\langle 1_\O\times U_\T^nf,1_\O\times g \rangle=0,\]
		that is,
		\[\lim_{n\to \infty,n\in {S}}\lk(\langle U_T^{(n,[n\beta])}f, g \rangle\cdot(1-\{n\beta\})+\langle U_T^{(n,[n\beta]+1)}f, g \rangle\cdot\{n\beta\}\re)=0.\]
		By the ergodicity of $(\O,\mathcal{C},m,R_\b)$, one has 
		\[\lim_{k\to \infty}\frac{1}{\#(\Lambda_k^{\vec{v}}(b))}\sum_{(m,n)\in\Lambda_k^{\vec{v}}(b)}|\langle U_T^{(m,n)}f,g\rangle|=0,\]
		which implies that $WM^{\mu,\vec{v}}(X)\subset \mathfrak{A}$.
		
		Next we prove that $\mathfrak{A}\subset WM^{\mu,\vec{v}}(X).$
		Let $f\in \mathfrak{A}$. We may assume that $\|f\|_2\neq 0$, because if $\|f\|_2=0$, it is clear that $f\in WM^{\mu,\vec{v}}(X)$.
		By contradiction, we suppose that $f\notin WM^{\mu,\vec{v}}(X)$. Then there exists $A\in\mathcal{B}_X$ with $\mu(A)>0$ such that
		$\mathbb{E}(f|\mathcal{K}_\mu^{\vec{v}})(x)\neq 0, \text{ for every } x\in A$, which implies that $\|\mathbb{E}(f|\mathcal{K}_\mu^{\vec{v}})\|_2>0.$ Since $f-\mathbb{E}(f|\mathcal{K}_\mu^{\vec{v}})\in WM^{\mu,\vec{v}}(X)$, and $WM^{\mu,\vec{v}}(X)\subset \mathfrak{A}$ which we have proved, it follows that
		\begin{equation}
			\begin{split}
				0=&\lim_{k\to \infty}\frac{1}{\#(\Lambda_k^{\vec{v}}(b))}\sum_{(m,n)\in\Lambda_k^{\vec{v}}(b)}|\langle U_T^{(m,n)}(f-\mathbb{E}(f|\mathcal{K}_\mu^{\vec{v}})),\mathbb{E}(f|\mathcal{K}_\mu^{\vec{v}})\rangle|\\
				=&\lim_{k\to \infty}\frac{1}{\#(\Lambda_k^{\vec{v}}(b))}\sum_{(m,n)\in\Lambda_k^{\vec{v}}(b)}|\langle U_T^{(m,n)}f,\mathbb{E}(f|\mathcal{K}_\mu^{\vec{v}})\rangle-\langle \mathbb{E}(f|\mathcal{K}_\mu^{\vec{v}}),\mathbb{E}(f|\mathcal{K}_\mu^{\vec{v}})\rangle|\\
			\end{split}
		\end{equation}
		Therefore, \[\lim_{k\to \infty}\frac{1}{\#(\Lambda_k^{\vec{v}}(b))}\sum_{(m,n)\in\Lambda_k^{\vec{v}}(b)}\langle U_T^{(m,n)}f,\mathbb{E}(f|\mathcal{K}_\mu^{\vec{v}})\rangle=\|\mathbb{E}(f|\mathcal{K}_\mu^{\vec{v}})\|_2^2\neq0,\]
		a contradiction. This shows that $\mathfrak{A}\subset WM^{\mu,\vec{v}}(X).$
	\end{proof}

	\section{Mean ergodic theorem}\label{sec:mean ergodic}
	The following result is a directional version of a mean ergodic theorem, which can be viewed as an extension of \cite{MR285690} to directional dynamical systems.
	\begin{thm}\label{thm:mean ergodic}
		Let  $(X,\mathcal{B}_X,\mu,T)$ be a $\mathbb{Z}^2$-m.p.s., $\vec{v}=(1,\beta)\in\mathbb{R}^2$ be a directional vector and $b\in(0,\infty)$. Then the following conditions are equivalent.
		\begin{itemize}
			\item[(a)]	$(X,\mathcal{B}_X,\mu,T)$ is  $\vec{v}$-weakly mixing.	
			\item[(b)]For any infinite subsequence $Q=\{(m_i,n_i)\}_{i=1}^\infty$ of $\Lambda_+^{\vec{v}}(b)$  with
			$\underline{D}_b(Q)>0$
			one has 
			\begin{equation}\label{66}
				\lim_{N\to \infty}\|\frac{1}{N}\sum_{i=1}^{N}U_T^{(m_i,n_i)}f-\int_{X}fd\mu\|_2=0
			\end{equation}
			for all $f\in L^2(X,\mathcal{B}_X,\mu)$.
		\end{itemize}
	\end{thm}
	\begin{proof}
		(a) $\Rightarrow$ (b). Without loss of generality, we may assume that $\{m_i\}_{i=1}^\infty$ is a strictly monotone increasing sequence, otherwise we use the method in the proof of Corollary \ref{cor3} to decompose  $\{m_i\}_{i=1}^\infty$ into sereval strictly monotone increasing  sequences. 
		
		For any $A\in\mathcal{B}_X$, one has 
		\begin{equation}\label{77}
			\begin{split}
				&\|\frac{1}{N}\sum_{i=1}^{N}U_T^{(m_i,n_i)}1_A-\mu(A)d\mu\|_2^2=\int_X|\frac{1}{N}\sum_{i=1}^{N}1_{T^{-(m_i,n_i)}A}-\mu(A)|^2d\mu\\
				=&\frac{1}{N^2}\sum_{i,j=1}^{N}\mu({T^{-(m_i-m_j,n_i-n_j)}A\cap A})-\mu(A)^2.
			\end{split}
		\end{equation}
		Since for each $i\in\{1,2,\ldots,N\}$, 
		$\beta m_i-b\leq n_i\leq \beta m_i+b$, it follows that for each $i,j\in\{1,2,\ldots,N\}$
		\[\beta( m_i-m_j)-2b\leq n_i-n_j\leq \beta (m_i-m_j)+2b.\]
		So $(m_i-m_j,n_i-n_j)\in\Lambda^{\vec{v}}(2b)$ for each $i,j\in\{1,2,\ldots,N\}$. Hence for each $j\in\{1,2,\ldots,N\}$, one has 
		\begin{equation}\label{88}
			\begin{split}
				&\sum_{i=1}^{N}|\mu({T^{-(m_i-m_j,n_i-n_j)}A\cap A})-\mu(A)^2|\\
				\leq &\sum_{(m,n)\in\Lambda^{\vec{v}}(2b)\cap([-m_N,m_N])\times\mathbb{Z}}|\mu({T^{-(m,n)}A\cap A})-\mu(A)^2|\\
				\leq &2\sum_{(m,n)\in\Lambda^{\vec{v}}_{m_N}(2b)}|\mu({T^{-(m,n)}A\cap A})-\mu(A)^2|.
			\end{split}
		\end{equation}
		Combining \eqref{77}, \eqref{88} and the fact that $(X,\mathcal{B}_X,\mu,T)$ is $\vec{v}$-weakly mixing, one has 
		\begin{align*}
			&\lim_{N\to \infty}\|\frac{1}{N}\sum_{i=1}^{N}U_T^{(m_i,n_i)}1_A-\mu(A)d\mu\|_2^2\\
			\leq&\lim_{N\to \infty}\frac{1}{N^2}\cdot \lk(N\cdot\lk( 2\sum_{(m,n)\in\Lambda^{\vec{v}}_{m_N}(2b)}|\mu({T^{-(m,n)}A\cap A})-\mu(A)^2|\re)\re)\\
			\leq &\lim_{N\to \infty}\frac{\#(\Lambda^{\vec{v}}_{m_N}(2b))}{N}\cdot\frac{2}{\#(\Lambda^{\vec{v}}_{m_N}(2b))} \sum_{(m,n)\in\Lambda^{\vec{v}}_{m_N}(2b)}|\mu({T^{-(m,n)}A\cap A})-\mu(A)^2|\\
			\leq &\limsup_{N\to \infty}\frac{\#(\Lambda^{\vec{v}}_{N}(2b))}{\#(Q\cap([0,N-1]\times\mathbb{Z}))}\\
			&\qquad\qquad\qquad\cdot\lim_{N\to \infty}\frac{2}{\Lambda^{\vec{v}}_{m_N}(2b)} \sum_{(m,n)\in\Lambda^{\vec{v}}_{m_N}(2b)}|\mu({T^{-(m,n)}A\cap A})-\mu(A)^2|\\
			= &\frac{1}{\underline{D}_{2b}(Q)}\cdot0=0.
		\end{align*}
		Since the set of all $f\in L^2(X,\mathcal{B}_X,\mu)$ for which \eqref{66} holds forms a closed subspace  of $L^2(X,\mathcal{B}_X,\mu)$ which contains the set of indicator function, it follows that this subspace is $L^2(X,\mathcal{B}_X,\mu)$ itself.
		
		(b) $\Rightarrow$ (a). Note that for any $b_1,b_2\in(0,\infty)$, if an infinite subset $Q$ of $\Lambda^{\vec{v}}(b_2)$ satisfies $\underline{D}_{b_2}(Q)>0$ then $\underline{D}_{b_1}(Q)>0$. Hence we only need to prove the case for $b=1$. Let $Q=\{(n,[n\beta])\}_{n=1}^\infty$. Then $Q$ is a subset of $\Lambda_k^{\vec{v}}(1)$ with $\underline{D}_{1}(Q)>0$ and we have the following claim.
		\item[\textbf{Claim.}] For any $B,C\in\mathcal{B}_X$ and $\delta,\epsilon>0$, there exists $K>0$ such that for any $N\geq K$ and $m\in\mathbb{N}$ the inequality \[|\mu({T^{-(m,[m\beta])}B\cap T^{-(n,[n\beta])}C})-\mu(B)\mu(C)|\geq\epsilon\]
		has  at most $\delta N$ solutions $n$ with $1\leq n\leq N.$
		
		Now supposing the claim holds, we prove that	$(X,\mathcal{B}_X,\mu,T)$ is  $\vec{v}$-weakly mixing.  Suppose, for the sake of contradiction, that $(X,\mathcal{B}_X,\mu,T)$ is not $\vec{v}$-weakly mixing. Then by Corollary \ref{cor3}, there exist $\epsilon>0$, $B,C\in\mathcal{B}_X$ and $Q\subset\mathbb{Z}$ with $d(Q)>0$ such that for any $n\in{Q}$
		\[|\mu({T^{-(n,[n\beta])}B\cap C})-\mu(B)\mu(C)|\geq\epsilon.\]
		So for any $K\in\mathbb{N}$, there exists $N\geq K$ such that 
		\[\frac{\#(Q\cap [0,N-1])}{N}>d(Q)/2\]
		and for any $n\in Q\cap[0,N-1]$
		\[|\mu({T^{-(n,[n\beta])}B\cap C})-\mu(B)\mu(C)|\geq\epsilon.\]
		That is, the above inequality has at least $(d(Q)/2)\cdot N$ solutions $n$ with $1\leq n\leq N$,
		which contradicts Claim. Therefore, $(X,\mathcal{B}_X,\mu,T)$ is  $\vec{v}$-weakly mixing.
		
		Finally, we only need to prove the claim holds. Indeed,  suppose the claim is not valid. Then  there exist $B,C\in\mathcal{B}_X$, $\delta,\epsilon>0$ and sequences $\{m_k\}_{k=1}^\infty$ and $\{N_k\}_{k=1}^\infty$ with $N_{k+1}/N_k\to\infty$ as $k\to \infty$ such that for any $k\in\mathbb{N}$
		\begin{equation}\label{99}
			\mu({T^{-(m_k,[m_k\beta])}B\cap T^{-(n,[n\beta])}C})\geq\mu(B)\mu(C)+\epsilon
		\end{equation}
		has at least $\delta N_k$ solutions $N_{k-1}\leq n\leq N_{k}$. Assume that for a fixed $k\in\mathbb{N}$, \eqref{99} holds for $n_1,\ldots,n_s$ with $N_{k-1}\leq n_i\leq N_{k}$ for each $i=1,\ldots,s$. Then by H\"older inequality, one has
		\begin{equation}\label{11}
			\begin{split}
				&\frac{1}{s^2}\sum_{i,j=1}^{s}\mu({T^{-(n_j,[n_j\beta])}B\cap T^{-(n_i,[n_i\beta])}B})\\
				=&\mu(B)^2+\frac{1}{s^2}\int_X\lk(\sum_{i=1}^{s}1_{T^{-(n_i,[n_i\beta])}B}-s\mu(B)\re)^2d\mu\\
				\geq&\mu(B)^2+\frac{1}{s^2}\int_X\lk(\sum_{i=1}^{s}1_{T^{-(n_i,[n_i\beta])}B}\cdot1_{T^{-(m_k,[m_k\beta])}C}-s\mu(B)\cdot1_{T^{-(m_k,[m_k\beta])}C}\re)^2d\mu\\
				=&\mu(B)^2+\frac{1}{s^2}\|\sum_{i=1}^{s}1_{T^{-(n_i,[n_i\beta])}B}\cdot1_{T^{-(m_k,[m_k\beta])}C}-s\mu(B)\cdot1_{T^{-(m_k,[m_k\beta])}C}\|_2^2\\
				\geq&\mu(B)^2+\frac{1}{s^2}\|\sum_{i=1}^{s}1_{T^{-(n_i,[n_i\beta])}B}\cdot1_{T^{-(m_k,[m_k\beta])}C}-s\mu(B)\cdot1_{T^{-(m_k,[m_k\beta])}C}\|_1^2\\
				\geq&\mu(B)^2+\frac{1}{s^2}\lk(\sum_{i=1}^{s}\mu({T^{-(n_i,[n_i\beta])}B\cap T^{-(m_k,[m_k\beta])}C})-s\mu(B)\mu(C)\re)^2\\
				\overset{\eqref{99}}\geq &\mu(B)^2+\epsilon^2.
			\end{split}
		\end{equation}
		Consider a subsequence of $\{(n,[n\beta])\}_{n=1}^\infty$ consisting of all solutions $n$ taken for all $k\in\mathbb{N}$, denoted by $Q:=\{(n_j,[n_j\beta])\}_{j=1}^\infty$. According to the construction, one has for each $k\in\mathbb{N}$
		\begin{equation}\label{333}
			\frac{\#(Q\cap([0,N_k]\times \mathbb{Z}))}{ N_k}\geq \delta
		\end{equation}
		and
		\begin{equation}\label{222}
			\begin{split}
				\frac{1}{\#(Q\cap([0,N_k]\times \mathbb{Z}))^2}&\sum_{(m,n),(m',n')\in Q\cap([0,N_k]\times \mathbb{Z})}\mu(T^{-(m,n)}B\cap T^{-(m',n')}B)\\
				&\qquad\qquad\qquad\qquad\qquad\qquad\qquad\overset{\eqref{11}}\geq\mu(B)^2+\epsilon^2.
			\end{split}
		\end{equation}
		Let 
		$d=\frac{1}{\lk(\lk(\sqrt{\frac{\mu(B)^2+\epsilon^2}{\mu(B)^2+\epsilon^2/2}}-1\re)\delta\re)}\text{ and } P=Q\cup\{(nd,[nd\beta])\}_{n=1}^\infty.$
		Then \begin{equation}\label{16}
			\liminf_{k\to\infty}\frac{\#(P\cap ([0,k-1]\times \mathbb{Z}))}{\#(\Lambda_k^{\vec{v}}(1))}>0
		\end{equation}
		and 
		\begin{align*}
			&\frac{1}{\#(P\cap([0,N_k]\times \mathbb{Z}))^2}\sum_{(m,n),(m',n')\in P\cap([0,N_k]\times \mathbb{Z}}\mu(T^{-(m,n)}B\cap T^{-(m',n')}B)\\
			\geq& \frac{1}{(\#(Q\cap([0,N_k]\times \mathbb{Z})+\frac{N_k}{d}))^2}\sum_{(m,n),(m',n')\in Q\cap([0,N_k]\times \mathbb{Z})}\mu(T^{-(m,n)}B\cap T^{-(m',n')}B)\\
			\overset{\eqref{222}}\geq& (\mu(B)^2+\epsilon^2)\frac{(\#(Q\cap([0,N_k]\times \mathbb{Z})))^2}{(\#(Q\cap([0,N_k]\times \mathbb{Z})+\frac{N_k}{d}))^2}\\
			=&(\mu(B)^2+\epsilon^2)\frac{1}{(1+\frac{N_k}{d\cdot(\#(Q\cap([0,N_k]\times \mathbb{Z}))})^2}\\
			\overset{\eqref{333}}\geq&(\mu(B)^2+\epsilon^2)\frac{1}{(1+\frac{N_k}{d\delta N_k})^2}\\
			=&\mu(B)^2+\epsilon^2/2.
		\end{align*}
		We enumerate the above $P=\{(s_l,[s_l\beta])\}_{l=1}^\infty$. Then by \eqref{77} and the above inequality, we obtain that 
		\begin{equation*}
			\begin{split}
				&\|\frac{1}{N}\sum_{i=1}^{N}U_T^{(s_l,[s_l\beta])}1_B-\mu(B)\|_2^2\\
				=&\frac{1}{\#(P\cap\Lambda_k^{\vec{v}}(2))^2}\sum_{(m,n),(m',n')\in P\cap\Lambda_k^{\vec{v}}(2)}\mu(T^{-(m,n)}B\cap T^{-(m',n')}B)-\mu(B)^2\\
				>&\epsilon^2/2.
			\end{split}
		\end{equation*}
		This implies that 
		\[\limsup_{N\to \infty}\|\frac{1}{N}\sum_{i=1}^{N}U_T^{(s_l,[s_l\beta])}1_B-\mu(B)\|_2>\epsilon,\]
		a  contradiction to (b).  Thus, we finish the proof of Theorem \ref{thm:mean ergodic}.
	\end{proof}

	\section{Relation between weak mixing and directional weak mixing}\label{sec:relation}
	In this section, we investigate the relation between weak mixing under $\mathbb{Z}^2$-actions and directional weak mixing.  For the definition of weak mixing under group actions, we can refer to \cite{BM}. In particular, we say a $\mathbb{Z}^2$-m.p.s. $(X,\mathcal{B}_{X},\mu,T)$  weak mixing if its Kronecker algebra $\mathcal{K}_{\mu}=\{X,\emptyset\}$.
	
	Recall that authors \cite{LX} proved that for a $\mathbb{Z}^2$-m.p.s. $(X,\mc{B}_X,\mu,T)$, $\mu$ has discrete spectrum if and only if it has directional discrete spectrum along two different directions. However, the following example shows that there exists a weakly mixing $\mathbb{Z}^2$-m.p.s. with the property that is neither $\vec{v}$-weakly mixing nor $\vec{w}$-weakly mixing, where $\vec{v},\vec{w}\in \mathbb{R}^2$ are two linearly independent direction vectors.
	\begin{exam}
		Denote by $(Y,2^Y,\mu)$ a measure space, where $Y=\{0,1\}$, $2^Y$ is the collection consisting of all subsets of $Y$ and the points $0,1$ have measure $1/2$. Let $$(X,\mathcal{B}_X,m)=\prod_{-\infty}^{\infty}(Y,2^Y,\mu).$$ Define $Id_X:X\to X$ by $Id_X(\{x_n\})=\{x_n\}$ and $T:X\to X$ by $T(\{x_n\})=\{y_n\}$ where $y_n=x_{n+1}$ for all $n\in\mb{Z}$, that is, $T$ is the two-sided $(\frac{1}{2},\frac{1}{2})$-shift.  Let $(X\times X,\mathcal{B}_X\times\mathcal{B}_X,m\times m)$ be the product measure space, $T_1=Id_X\times T$ and $T_2=T\times Id_X$. Then we define a $\mb{Z}^2$-action $\widetilde{T}$  on $X\times X$ by $\widetilde{T}^{(m,n)}=T_1^{(m+n)}T_2^{(m-n)}$ for all  $(m,n)\in \mb{Z}^2$.
		Let $\widetilde{X}=X\times X$, $\mathcal{B}_{\widetilde{X}}=\mathcal{B}_{X\times X}$ and $\widetilde{\mu}=m\times m$. Then we obtain a $\mathbb{Z}^2$-m.p.s. $(\widetilde{X},\mathcal{B}_{\widetilde{X}},\widetilde{\mu},\widetilde{T}).$
		Take $\vec{v}=(1,-1)$ and $\vec{w}=(1,1)$.
		
		It is easy to check it is weakly mixing.
		Now, we show that $(\widetilde{X},\mathcal{B}_{\widetilde{X}},\widetilde{\mu},\widetilde{T})$ is neither $\vec{v}$-weakly mixing nor $\vec{w}$-weakly mixing.
		For any $B\in \mathcal{B}_X$, we take $X\times B\in \mathcal{B}_{\widetilde{X}}$. Then for $b=1/3$,  one has
		$$\overline{\{U_{\widetilde{T}}^{(m,n)}1_{X\times B}:(m,n)\in\Lambda^{\vec{v}}(1/3) \}}=\overline{\{U_{\widetilde{T}}^{(n,-n)}1_{X\times B}:n\in\mb{Z}\}}=\{1_{X\times B}\}$$
		is a compact subset of $L^2(\widetilde{X}, \mc{B}_{\widetilde{X}},\widetilde{\mu})$, which implies that $(\widetilde{X},\mathcal{B}_{\widetilde{X}},\widetilde{\mu},\widetilde{T})$ is not $\vec{v}$-weakly mixing. Similarly, we can show $(\widetilde{X},\mathcal{B}_{\widetilde{X}},\widetilde{\mu},\widetilde{T})$ is not $\vec{w}$-weak mixing.
	\end{exam} 
	\begin{rem}
		In \cite[Example 2.9]{2022arXiv220306710R}, authors provided an example to show that there exists a weakly mixing $\mathbb{Z}^2$-m.p.s. which is not directional weak mixing along any rational direction under their definition. It is easy to check that it also holds for our definition.
	\end{rem}

	Conversely, we have the following result, which is obtained by the definition.
	\begin{prop}
		Let  $(X,\mathcal{B}_X,\mu,T)$ be a $\mathbb{Z}^2$-m.p.s. If there exists a direction $\vec{v}\in\mathbb{R}^2$ such that $(X,\mathcal{B}_X,\mu,T)$ is $\vec{v}$-weakly mixing, then it is weakly mixing.
	\end{prop}
	
	In order to further study the relation between directional weak mixing and weak mixing, we begin with a combinatorial result in \cite{LX}.
	\begin{lem}\label{c}
		Let  $\vec{v}=(1,\beta_1), \vec{w}=(1,\beta_2)\in \mb{R}^2$ be two directions with $\beta_1\neq \beta_2$. Then $$\mb{Z}^2=\Lambda^{\vec{v}}(b)+\Lambda^{\vec{w}}(b)$$ for any $b> 4([|\beta_1-\beta_2|]+1)$, where $$\Lambda^{\vec{v}}(b)+\Lambda^{\vec{w}}(b)=\{(m_1+m_2,n_1+n_2):(m_1,n_1)\in \Lambda^{\vec{v}}(b)\text{ and } (m_2,n_2)\in \Lambda^{\vec{w}}(b)\}.$$
	\end{lem} 
	
	With the help of the above lemma, we are able to prove following result.
	\begin{thm}	\label{thm12}
		Let  $(X,\mathcal{B}_X,\mu,T)$ be a $\mathbb{Z}^2$-m.p.s. Then the following statements are equivalent.
		\begin{itemize}
			\item[(a)] There exist two directions $\vec{v}=(1,\beta_1),\vec{w}=(1,\beta_2)\in \mb{R}^2$ with $\beta_1\neq \beta_2$ such that $$\mathcal{K}^{\vec{v}}_{\mu}\cap\mathcal{K}^{\vec{w}}_{\mu}=\{X,\emptyset\}.$$
			\item[(b)] $(X,\mathcal{B}_X,\mu,T)$ is  weakly mixing.
		\end{itemize}
		\begin{proof}	
			(b) $\Rightarrow$ (a). Suppose, for the sake of contradiction, that  $\mathcal{K}^{\vec{v}}_{\mu}\cap\mathcal{K}^{\vec{w}}_{\mu}\neq \{X,\emptyset\}.$ Then take $B\in\mathcal{K}^{\vec{v}}_{\mu}\bigcap\mathcal{K}^{\vec{w}}_{\mu}$ with $0<\mu(B)<1$. We will show that $\overline{\{U_T^{(m,n)}1_B:(m,n)\in \mathbb{Z}^2\}}$
			is a compact subset of $L^2(X,\mc{B}_X,\mu)$, which implies that  $(X,\mathcal{B}_X,\mu,T)$ is not weakly mixing. This is a contradiction to (b). In fact, by Lemma \ref{c}, it suffices to prove that $$\mathcal{R}_b:=\overline{\{U_T^{(m,n)}1_B:(m,n)\in \Lambda^{\vec{v}}(b)+\Lambda^{\vec{w}}(b)\}}$$ is a compact subset of $L^2(X,\mathcal{B}_X,\mu)$. So
			$$\mathcal{P}_b:=\overline{\{U_T^{(m,n)}1_B:(m,n)\in \Lambda^{\vec{v}}(b)\}} \text{ and } \mathcal{Q}_b:=\overline{\{U_T^{(m,n)}1_B:(m,n)\in \Lambda^{\vec{w}}(b)\}}$$  are  compact subsets of $L^2(X,\mathcal{B}_X,\mu)$.
			For any $\epsilon>0$, let $$\{(m_i,n_i)\}_{i=1}^s\subset \Lambda^{\vec{v}}(b)\text{ and } \{(u_j,v_j)\}_{j=1}^s\subset \Lambda^{\vec{w}}(b)$$ be $\epsilon/2$-nets of $\mathcal{P}_b$ and $\mathcal{Q}_b$, respectively. Hence for any $(p_1,q_1)\in \Lambda^{\vec{v}}(b)$ and $(p_2,q_2)\in \Lambda^{\vec{w}}(b)$, we conclude
			$$\|U_T^{(p_1,q_1)}1_{B}-U_T^{(m_i,n_i)}1_{B}\|_2<\epsilon/2 \text{ and } \|U_T^{(p_2,q_2)}1_{B}-U_T^{(u_j,v_j)}1_{B}\|_2<\epsilon/2$$
			for some $i,j\in\{1,\ldots,s\}$.
			Therefore
			\begin{equation}\label{5}
				\begin{split}
					&\|U_T^{(p_1+p_2,q_1+q_2)}1_{B}-U_T^{(m_i+u_j,n_i+v_j)}1_{B}\|_2\\
					\leq& \|U_T^{(p_1,q_1)}1_{T^{-(p_2,q_2)}B}-U_T^{(m_i,n_i)}1_{T^{-(p_2,q_2)}B}\|_2\\
					&+\|U_T^{(p_2,q_2)}1_{T^{-(m_i,n_i)}B}-U_T^{(u_j,v_j)}1_{T^{-(m_i,n_i)}B}\|_2\leq \epsilon.
				\end{split}
			\end{equation}
			It follows from \eqref{5} that $$\Theta_b:=\{(m_i+u_j,n_i+v_j):1\leq i,j\leq s\}$$ is a finite $\epsilon$-net of $\mathcal{R}_b$ in $L^2(X,\mathcal{B}_X,\mu)$, which implies that $\mathcal{R}_b$ is a compact subset of $L^2(X,\mathcal{B}_X,\mu)$.

			(a) $\Rightarrow$ (b). If $(X,\mathcal{B}_X,\mu,T)$ is not weakly mixing, then there exists $B\in\mathcal{K}_{\mu}$ with $0<\mu(B)<1$ such that $\overline{\{U_T^{(m,n)}1_B:(m,n)\in \mathbb{Z}^2\}}$ is a compact subset of $L^2(X,\mathcal{B}_X,\mu)$. Since  $\mathcal{P}_b$ and $\mathcal{Q}_b$ are closed subsets of $\overline{\{U_T^{(m,n)}1_B:(m,n)\in \mathbb{Z}^2\}}$, it follows that $\mathcal{P}_b$ and $\mathcal{Q}_b$ are compact subsets of $L^2(X,\mathcal{B}_X,\mu)$, which implies that $\mathcal{K}^{\vec{v}}_{\mu}\cap\mathcal{K}^{\vec{w}}_{\mu}\neq \{X,\emptyset\},$
			a contradiction. Thus, $(X,\mathcal{B}_X,\mu,T)$ is  weakly mixing.
			The proof of this theorem is completed.
		\end{proof}
	\end{thm}

	\section*{Acknowledgement}
	The author would like to thank Wen Huang, Leiye Xu,  Xiangdong Ye and Tao Yu for making many valuable suggestions. Chunlin Liu was partially supported by NNSF of China (12090012, 12031019, 12090010).

	\bibliographystyle{acm}

\end{document}